\documentclass[letterpaper, 12pt]{article}

\usepackage{mhfig} 
\usepackage{shuffle}

\def\2{{\mhpastefig[2/3]{tree11}}}
\def\3{{\!\mhpastefig[1/2]{tree12}}}
\usepackage[latin1]{inputenc}  
\usepackage{amssymb} 
\usepackage{amsmath}
\usepackage{amsthm}
\usepackage{txfonts}
\usepackage{subfigure}
\usepackage{graphicx}
\usepackage{bm} 
\usepackage{bbm}
\usepackage[normalem]{ulem}
\usepackage{url}
\usepackage{soul}
\usepackage{exscale}
\usepackage[usenames,dvipsnames]{xcolor} 
\usepackage{hyperref}
\usepackage{verbatim}
\usepackage[usenames,dvipsnames]{xcolor}
\usepackage{comment}
\usepackage{epsfig}

\newtheorem{definition}{Definition}[section]
\newtheorem{lemma}{Lemma}[section]
\newtheorem{remark}{Remark}[section]
\newtheorem{example}{Example}[section]
\newtheorem{theorem}{Theorem}[section]
\newtheorem{cor}{Corollary}[section]

\numberwithin{equation}{section}
\numberwithin{figure}{section}

\newcommand{\assign}{:=}
\newcommand{\backassign}{=:}

\newcommand{\tmop}[1]{\ensuremath{\operatorname{#1}}}

\newcommand{\E}{\mathbb{E}}
\newcommand{\R}{\mathbb{R}}

\newcommand{\beas}{\begin{eqnarray*}}
\newcommand{\eeas}{\end{eqnarray*}}
\newcommand{\bea}{\begin{eqnarray}}
\newcommand{\eea}{\end{eqnarray}}
\newcommand{\ben}{\begin{enumerate}}
\newcommand{\een}{\end{enumerate}}

\newcommand{\dm}{\diamond}

\newcommand{\tmF}{ {\mathbb{F}}}
\newcommand{\mG}{ {\mathbb{G}}}

\newcommand{\Y}{Y}

\newcommand{\cA}{\mathcal{A}}

\newcommand{\cF}{\mathcal{F}}
\newcommand{\mF}{\mathbb{F}}
\newcommand{\mK}{\mathbb{K}}
\newcommand{\mT}{\mathbb{T}}

\newcommand{\eps}{\epsilon}
\newcommand{\RR}{\mathbb{R}}

\newcommand{\bi}{\begin{itemize}}
\newcommand{\ei}{\end{itemize}}
\newcommand{\beq}{\begin{equation}}
\newcommand{\eeq}{\end{equation}}
\newcommand{\bv}{\begin{verbatim}}
\newcommand{\ev}{\end{verbatim}}

\newcommand{\ee}[1]{\ensuremath{\mathbb{E}\left[{#1}\right]}}

\newcommand{\eet}[1]{\mathbb{E}_t\left[#1\right]}
\newcommand{\eef}[1]{\mathbb{E}_t\left[#1\right]}

\newcommand{\angl}[1]{\langle{#1}\rangle}

\usepackage{tikz}
\usepackage{mhequ} 
\usetikzlibrary{snakes}
\usetikzlibrary{decorations}
\usetikzlibrary{positioning}
\usetikzlibrary{shapes}
\usetikzlibrary{external}
\makeatletter
\pgfdeclareshape{crosscircle}
{
\inheritsavedanchors[from=circle] 
\inheritanchorborder[from=circle]
\inheritanchor[from=circle]{north}
\inheritanchor[from=circle]{north west}
\inheritanchor[from=circle]{north east}
\inheritanchor[from=circle]{center}
\inheritanchor[from=circle]{west}
\inheritanchor[from=circle]{east}
\inheritanchor[from=circle]{mid}
\inheritanchor[from=circle]{mid west}
\inheritanchor[from=circle]{mid east}
\inheritanchor[from=circle]{base}
\inheritanchor[from=circle]{base west}
\inheritanchor[from=circle]{base east}
\inheritanchor[from=circle]{south}
\inheritanchor[from=circle]{south west}
\inheritanchor[from=circle]{south east}
\inheritbackgroundpath[from=circle]
\foregroundpath{
\centerpoint%
\pgf@xc=\pgf@x%
\pgf@yc=\pgf@y%
\pgfutil@tempdima=\radius%
\pgfmathsetlength{\pgf@xb}{\pgfkeysvalueof{/pgf/outer xsep}}%
\pgfmathsetlength{\pgf@yb}{\pgfkeysvalueof{/pgf/outer ysep}}%
\ifdim\pgf@xb<\pgf@yb%
  \advance\pgfutil@tempdima by-\pgf@yb%
\else%
  \advance\pgfutil@tempdima by-\pgf@xb%
\fi%
\pgfpathmoveto{\pgfpointadd{\pgfqpoint{\pgf@xc}{\pgf@yc}}{\pgfqpoint{-0.707107\pgfutil@tempdima}{-0.707107\pgfutil@tempdima}}}
\pgfpathlineto{\pgfpointadd{\pgfqpoint{\pgf@xc}{\pgf@yc}}{\pgfqpoint{0.707107\pgfutil@tempdima}{0.707107\pgfutil@tempdima}}}
\pgfpathmoveto{\pgfpointadd{\pgfqpoint{\pgf@xc}{\pgf@yc}}{\pgfqpoint{-0.707107\pgfutil@tempdima}{0.707107\pgfutil@tempdima}}}
\pgfpathlineto{\pgfpointadd{\pgfqpoint{\pgf@xc}{\pgf@yc}}{\pgfqpoint{0.707107\pgfutil@tempdima}{-0.707107\pgfutil@tempdima}}}
}
}
\makeatother

\def\A{\tikz[baseline=-2.8,scale=0.15]{\node[A] {};}}
\def\X{\tikz[baseline=-2.8,scale=0.15]{\node[X] {};}}
\def\Z{\tikz[baseline=-2.8,scale=0.15]{\node[Z] {};}}

\def\AAd{\tikz[baseline=-1,scale=0.15]{\draw (-1,1) node[A] {} -- (0,0) node[not] {} -- (1,1) node[A] {};}} 
\def\AXd{\tikz[baseline=-1,scale=0.15]{\draw (-1,1) node[A] {} -- (0,0) node[not] {} -- (1,1) node[X] {};}} 
\def\AZd{\tikz[baseline=-1,scale=0.15]{\draw (-1,1) node[A] {} -- (0,0) node[not] {} -- (1,1) node[Z] {};}} 
 
\def\ZZd{\tikz[baseline=-1,scale=0.15]{\draw (-1,1) node[Z] {} -- (0,0) node[not] {} -- (1,1) node[Z] {};}}

\def\AAdAAdd{\tikz[baseline=-1,scale=0.15]{\draw (0,0) node[not] {} -- (-1,1) node[not] {};
\draw (0,0) -- (1,1) node[not] {};
\draw (-1,1) -- (-1.5,2.5) node[A] {};
\draw (-1,1) -- (-0.5,2.5) node[A] {};
\draw (1,1) -- (0.5,2.5) node[A] {};
\draw (1,1) -- (1.5,2.5) node[A] {};}}

\def\AAdAdAd{\tikz[baseline=-1,scale=0.15]{\draw (0,0) node[not] {} -- (-1,1) node[not] {}
-- (-2,2) node[not]{} -- (-3,3) node[A]  {};
\draw (0,0) -- (1,1) node[A] {};
\draw (-1,1) -- (0,2) node[A] {};
\draw (-2,2) -- (-1,3) node[A] {};}}

\def\AXdAdAd{\tikz[baseline=-1,scale=0.15]{\draw (0,0) node[not] {} -- (-1,1) node[not] {}
-- (-2,2) node[not]{} -- (-3,3) node[A]  {};
\draw (0,0) -- (1,1) node[A] {};
\draw (-1,1) -- (0,2) node[X] {};
\draw (-2,2) -- (-1,3) node[A] {};}}

\def\AAdAd{\tikz[baseline=-1,scale=0.15]{\draw (0,2) node[A] {} -- (-1,1) ; \draw (-2,2)  node[A] {} -- (-1,1) ; \draw (-1,1)  node[not] {} -- (0,0); 
\draw (0,0) node[not] {}  -- (1,1) node[A] {};}}

\def\AAdAddAAd{\tikz[baseline=-1,scale=0.15]{
\draw (0,0) node[not] {} -- (-1,1) node[not] {}
-- (-2,2) node[not]{} -- (-3,3) node[A]  {};
\draw (-1,1) -- (-.5,2) node[A] {};
\draw (-2,2) -- (-1,3) node[A] {};
\draw (0,0) -- (1,1) node[not] {};
\draw (1,1) -- (2,2) node[A] {};
\draw (1,1) -- (.5,2) node[A] {};
}}

\def\AAdAdAdAd{\tikz[baseline=-1,scale=0.15]{
\draw (0,0) node[not] {} -- (-1,1) node[not] {}
-- (-2,2) node[not]{} -- (-3,3) node[not]  {} 
-- (-4,4) node[A]{};
\draw (0,0) -- (1,1) node[A] {};
\draw (-1,1) -- (0,2) node[A] {};
\draw (-2,2) -- (-1,3) node[A] {};
\draw (-3,3) -- (-2,4) node[A] {};}}

\def\AAdAAddAd{\tikz[baseline=-1,scale=0.15]{
\draw (1,0) node[not] {} -- (0,1) node[not] {};
\draw (1,0) node[not] {} -- (2,1) node[A] {};
\draw (0,1) node[not] {} -- (-1,2) node[not] {};
\draw (0,1) -- (1,2) node[not] {};
\draw (-1,2) -- (-1.5,3.5) node[A] {};
\draw (-1,2) -- (-0.5,3.5) node[A] {};
\draw (1,2) -- (0.5,3.5) node[A] {};
\draw (1,2) -- (1.5,3.5) node[A] {};}}

\colorlet{symbols}{blue!90!black}
\colorlet{testcolor}{green!60!black}
\colorlet{connection}{red!30!black}

\tikzset{
root/.style={circle,fill=black!50,inner sep=0pt, minimum size=3mm},
    dot/.style={circle,fill=black,inner sep=0pt, minimum size=1.5mm},
    A/.style={very thin,circle,fill=black!10,draw=black,inner sep=0pt,minimum size=1.2mm}, 
    X/.style={very thin,crosscircle,fill=black!10,draw=black,inner sep=0pt,minimum size=1.2mm}, 
    Z/.style={very thin,circle,fill=black!100,draw=black,inner sep=0pt,minimum size=1.2mm}, 
not/.style={thin,circle,fill=symbols,draw=connection,fill=connection,inner sep=0pt,minimum size=0.35mm},
	>=stealth,
    }

\begin{document}

\title{Forests, cumulants, martingales}
\author{
Peter K. Friz\footnote{TU and WIAS Berlin, friz@math.tu-berlin.de (corresponding author)}, Jim Gatheral\footnote{Baruch College, CUNY, jim.gatheral@baruch.cuny.edu} \ and Rado\v{s} Radoi\v{c}i\'c\footnote{Baruch College, CUNY, rados.radoicic@baruch.cuny.edu}} 
\date{}
\maketitle\thispagestyle{empty}
\begin{abstract}
This work is concerned with forest and cumulant type expansions of general random variables on a filtered probability spaces.  
We establish a ``broken exponential martingale" expansion that generalizes and unifies the exponentiation result of Al{\`o}s, Gatheral, and Radoi{\v c}i\'c (SSRN'17; \cite{alos2020exponentiation}) and the cumulant recursion formula of Lacoin, Rhodes, and Vargas (arXiv; \cite{lacoin2019probabilistic}). 
Specifically, we exhibit the two previous results as lower dimensional projections of the same generalized forest expansion, subsequently related by forest reordering.  Our approach also leads to sharp integrability conditions  for validity of the cumulant formula, as required by many of our examples, including  iterated stochastic integrals, L\'evy area, Bessel processes, KPZ with smooth noise, Wiener-It\^o chaos and ``rough'' stochastic (forward) variance models. 
\medskip

\end{abstract}


\thispagestyle{empty}

\date{}

\section{Introduction}

\subsection{Statements of main results}

Consider a filtered probability space $(\Omega, \cF_T, (\cF_t)_{0 \le  t \le T};\mathbb{P})$, on which all martingales admit a continuous version.
(It\^o's representation theorem, e.g. \cite[Ch.V.3.]{revuz2013continuous}, states that this holds true for Brownian filtration which covers all situations we have in mind.) 
Throughout, $T \in (0,\infty]$ should be thought of as a fixed parameter. 

Let $A_T$ be $\cF_T$-measurable. Define, assuming sufficient integrability, 
$$ X_t := \log \mathbb{E}_t e^{A_T}, \ \ \ Y_t :=  \mathbb{E}_t A_T \ . $$
By construction, $X, Y$ have equal terminal value $X_T = Y_T = A_T$, and $e^X, Y$ are martingales.
Motivated by financial applications, in  \cite{alos2020exponentiation} an $\mathbb{F}$ {\it (forest) expansion} was given of the form\footnote{Corollary 3.1 in \cite{alos2020exponentiation} is an expansion of the characteristic function, with $z=i\xi$.}
\[ 
\mathbb{E}_t e^{z A_T} =  \mathbb{E}_t e^{z X_T} =  e^{z X_t + \tfrac12 z (z -1)  \mathbb{E}_t \langle X \rangle_{t,T} +  \cdots} \equiv
\exp \left( z X_t +  \tfrac12 z (z -1) (X \dm X)_t(T)  + \sum_{k \ge 2} \mathbb{F}_t^k (T;z) \right) \]
with quadratic recursion for the $\mathbb{F}$'s, homogenous in $X$ but not in $z$, representable as forests. 
But $A_T$ is also the terminal value of the martingale $Y$ so that 
\[
\mathbb{E}_t e^{z A_T} = \mathbb{E}_t e^{z Y_T} = e^{z Y_t + \tfrac12 z^2 \mathbb{E}_t Y^2_{t,T} + \tfrac{1}{3!} z^3 \mathbb{E}_t Y^3_{t,T} + \tfrac{1}{4!} z^4 ( \mathbb{E}_t Y^4_{t,T}
- 3(\mathbb{E}_t Y^2_{t,T})^2) + \cdots}  \equiv \exp \left( z Y_t + \sum_{n \ge 2} \mathbb{K}_t^n (T;z) \right) , \]
the (time-$t$) conditional $\mathbb{K}$ {\it (cumulant}, German: {\it Kumulanten) expansion} of $A_T = Y_T$. A somewhat similar quadratic recursion for the $\mathbb{K}$'s, homogenous in $z$ (equivalently: $Y$) was later 
obtained in \cite{lacoin2019probabilistic}, stated in the unconditional case, and motivated by applications in QFT,
and independently in a first version of this paper, when revisiting the convergence
properties of the $\mathbb{F}$-series. (Initially, the present authors were unaware of \cite{lacoin2019probabilistic}, whereas the authors of \cite{lacoin2019probabilistic} were unaware of
\cite{alos2020exponentiation}.\footnote{Preprint of \cite{alos2020exponentiation} posted on SSRN in 2017. We much share with the authors of \cite{lacoin2019probabilistic} our surprise that such recursions had not been discovered decades earlier.}) %
We note that the $\mF$-expansion was left as formal expansion in \cite{alos2020exponentiation}, whereas validity of the $\mK$ recursion of \cite{lacoin2019probabilistic} was only shown under a stringent integrability condition which rules out virtually all examples discussed later on. The main theorem of this paper is a $\mathbb{G}$ {\it (generalized forest) expansion}, which contains both $\mathbb{F}$- and $\mathbb{K}$-expansion as special cases, together with optimal integrability conditions for convergence. Our arguments are also well adapted to further localization, as seen in points (i) in Theorems \ref{thm:mainintroG}  and \ref{thm:mainintroK} below.

\begin{definition} \label{def:diamond}
Given two continuous semimartingales $A,B$ with integrable covariation process $\langle A , B \rangle$, the diamond product\footnote{Warning. Our diamond product is (very) different from the Wick product, e.g. Ch.III of \cite{janson1997gaussian}.} of $A$ and $B$ is another continuous semimartingale given by, writing  $\angl{A,B}_{t,T}$ for the difference of  $\angl{A,B}_T$ and $\angl{A,B}_t\, $,
$$
     (A \dm B)_t (T) := \eef{\langle A , B \rangle_{t,T}} = \eef{\langle A , B \rangle_{T}} -\langle A , B \rangle_t \; .
$$
\end{definition} 
 
Here and below we say that $A_T$ has exponential moments, if $\mathbb{E} e^{x A_T} < \infty$ for $x \in \mathbb{R}$ in a neighbourhood of zero. This of course implies that $A_T$ has moments of all orders: $A_T \in L^N$, for any $N \in \mathbb{N}$.
\begin{theorem}[$\mG$ expansion, a.k.a. broken exponential martingale] \label{thm:mainintroG} 

Let $A_T$ be a real-valued, $\cF_T$-measurable random variable.\\
(i) Assume $A_T$ has moments of all orders. With $z=(z_1, z_2) \in (i \,\mathbb{R})^2$, the following asymptotic expansion for the joint c.f. of $Y=\mathbb{E}_\bullet A_T$ and its quadratic variation holds, 
\begin{equation}  \label{eq:AsymCFnew}
\log \eef{e^{z_1 Y_T + z_2 \langle Y \rangle_T}} \sim 
z_1 Y_t + z_2 \langle Y \rangle_t 
+  \sum_{k\ge2}   \mG_t^k (T;z)   \qquad \text{ as $z \to 0$} \;,
  \end{equation}

\begin{equation}  \label{eq:Grec}
\mG^2  = \left( \frac{1}{2} z_1^2+ z_2\right)
(Y \dm Y)_t(T)  
\ \ \ \text{ and } \ \ \ \forall k > 2:  
\ \ \mG^k = \frac{1}{2} \sum_{j = 2}^{k -2} \mG^{k - j} \dm \mG^j +  (z_1 \,Y  \dm \mG^{k - 1}) \,.
\end{equation} 
(ii) If $A_T$ has exponential moments, (\ref{eq:AsymCFnew}) can be strengthened to equality, with a.s. absolutely convergent sum $\Lambda := \Lambda^T_t := \Lambda^T_t (z) := \sum_{k\ge2}  \mG_t^k (T;z)$ on the right-hand side, for $z \in \mathbb{C}^2$ with $|z| < \rho_t (\omega)$, a.s. strictly positive.

\noindent (iii)  For the multivariate case, with  $Y^i=\mathbb{E}_\bullet A^i_T, \ i=1,\dots,d$, it suffices to replace
$$
z_1 Y_T \rightsquigarrow  z_{1;i} Y^i_T, \quad  z_2 \langle Y \rangle_T \rightsquigarrow  z_{2;j,k} \langle Y^j,Y^k \rangle_T, \quad %
\mG^2   \rightsquigarrow \left( \frac{1}{2} z_{1;i} z_{1;j}+ z_{2;i,j}\right)
(Y^i \dm Y^j)_t(T)  
\, ,
$$
with summation over repeated indices. 
\end{theorem}

With $z=(z_1,-z_1/2)$, the $\mG$-recursion becomes precisely the $\mF$-recursion, equ. (3.1) in \cite{alos2020exponentiation}, whereas the case $z=(z_1,0)$ yields the $\mK$ (cumulant) recursion, equ. (3.4),(3.9) in \cite{lacoin2019probabilistic} as seen in (\ref{eq:Krec}) below. We should note that the change-of-measure based derivation of Lacoin et al. was given under stringent integrability assumption (a $L^\infty$ bound on $\langle Y \rangle_T$, hence Gaussian concentration of $Y_T = A_T$), though the authors (correctly) suspect validity of the cumulant recursion in greater generality. Below we achieve this under optimal conditions: conveniently, our proof of part (ii) below (found as such in the first arXiv version of this paper; now seen as special case of the proof of (ii) above), different from \cite{lacoin2019probabilistic}, applies directly under finite exponential moments, a necessary condition for the cumulant generating function to exist. By a careful localisation argument (alternatively: Hermite polynomials approach), we can further show, part (i) below, that the cumulant recursion is valid, as a {\it finite recursion}, under a matching (optimal) {\it finite moment} assumption. (Existence of the first $N$ moments is equivalent to existence of the first $N$ cumulants.) In part (iii), we give the multivariate formulation.

\begin{theorem} \label{thm:mainintroK}  
(i) Let $A_T$ be $\cF_T$-measurable with $N \in \mathbb{N}$ finite moments. Then the recursion
\begin{equation}  \label{eq:Krec}
\mK^{1}_t (T) :=  \eef{A_T} \ \ \ \text{ and } \ \ \ \forall n > 0:   
\ \ \mK^{n+1}_t (T)  = \frac{1}{2} \sum_{k=1}^n   ( \mK^{k} \dm \mK^{n+1-k} )_{t} (T)   \
\end{equation} 
is well-defined up to $\mK^{N}$ and, for $z \in i \, \R$,
$$
  \log \eef{e^{z A_T}}    =  \sum_{n = 1}^N z^n  \mK^{n}_t(T) + o (|z|^N) \quad \text{ as $z \to 0$} \; ,
$$
which identifies $n! \times \mK_t^{n}(T) $  as the (time $t$-conditional) $n$.th cumulant of $A_T$.
If $A_T$ has moments of all orders, we have the asymptotic expansion,
\begin{equation}  \label{eq:AsymCF}
\log \eef{e^{z A_T}} \sim \sum_{n = 1}^{\infty} z^{n} \mK^n_t  \ \ \ \text{ as $z \rightarrow 0$} \;.
  \end{equation}
\noindent (ii) If $A_T$ has exponential moments, so that its (time $t$-conditional) mgf 
$ \eef{  e^{x A_T}}$ is a.s. finite for $x\in \R $ in some neighbourhood of zero, then there exist a maximal convergence radius $\rho =\rho_t (\omega) \in (0,\infty]$ a.s. such that for all  $z \in \mathbb{C}$ with $|z| < \rho$,
 \begin{equation}  \label{eq:ConvCGF}
  \log \eef{e^{z A_T}} = \sum_{n = 1}^\infty z^{n} \mK^n_t \;. 
  \end{equation} 
\noindent (iii)  In the multivariate case, with  $Y^i=\mathbb{E}_\bullet A^i_T, \ i=1,\dots,d$ replace $\mK^n \rightsquigarrow \mK^{(n)}$, $n$-tensors (over $\mathbb{R}^d$), with tensorial interpretation of the diamond product in (\ref{eq:Krec}), and substitute in the expansion 
$$ z^{n} \mK^n_t \rightsquigarrow \langle z^{\otimes n }, \mK^{(n)} \rangle =  z_{i_1} \cdots  z_{i_n} \mK^{(n);i_1,\dots,i_n} \; . $$
The multivariate ($t$-conditional) cumulants of $A_T$ are then precisely given by $n! \mK_t^{(n);i_1,\dots,i_n} $.
   \end{theorem}

\subsection{Tree, forests and reordering} \label{sec:trees}

\tikzexternaldisable 

Both quadratic recursions (\ref{eq:Grec}) and (\ref{eq:Krec}) lead to binary trees, with the (commutative, non-associated) diamond product represented by {\it root joining}, 
$$
                \tau_1 \diamond \tau_2 = \tau_2 \diamond \tau_1 = \includegraphics[scale=0.04]{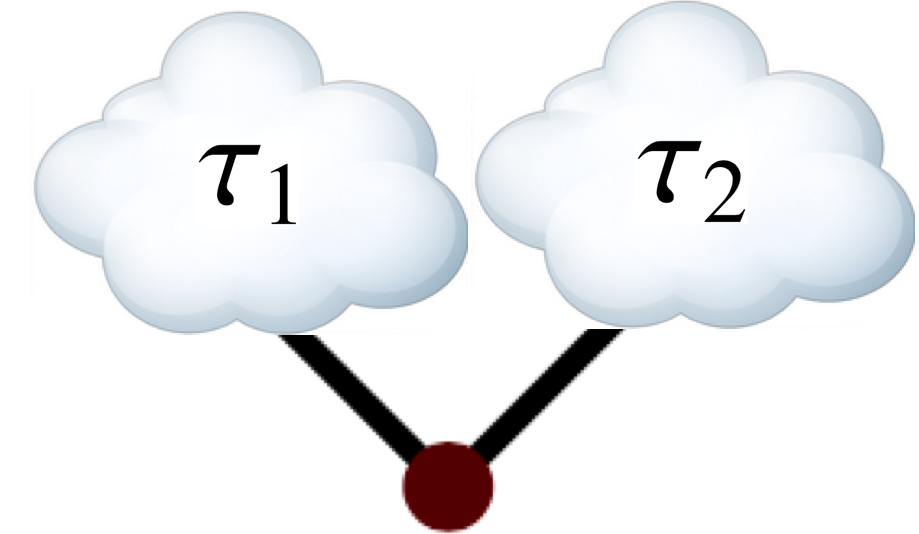} \quad .
$$
where we agree to set $Y \equiv \A \,$, interpreted as single leaf. In the $2$-variate case we can write $(Y^1, Y^2) = (\, \A \,, \X \,)$, for the $d$-variate case use labels or colors. For better readability, write 
$$
 (z_1,z_2) \rightsquigarrow (a,b), \qquad z_1 Y_T + z_2 \langle Y \rangle_T  \rightsquigarrow a Y_T + b \langle Y \rangle_T \;,
$$
in (\ref{eq:Grec}).The tree formalism is extremely convenient when it comes to {doing} explicit computations and also explains an interesting relation between {the} $\mG$-recursion and the $\mK$-recursion.  {Specifically, we see that the $\mG$-recursion is equalivent to the  $2$-variate $\mK$-recursion applied to $A_T = (Y_T, \langle Y \rangle_T)$ {\it after forest reordering}. This procedure has moreover the important effect of resolving infinite cancellations present in the $2$-variate $\mK$-expansion, as may be {seen} by applying {it} to the exponential martingale case ($b = - a^2/2$).
The first few $\mG$-forests are then spelled out as folllows: 
\begin{eqnarray}
\mG^2  &=&  (\tfrac{1} 2 a^2+b)\, \AAd \nonumber\\
\mG^3 &=& a\, (\tfrac 12 a^2+b)\, \AAdAd \nonumber\\
\mG^4 &=& \tfrac1{2} (\tfrac 12 a^2+b)^2\,\AAdAAdd+ 
a^2\, (\tfrac 12 a^2+b)\,  \AAdAdAd
\nonumber\\
\mG^5 
&=& a\, (\tfrac 12 a^2+b)^2\,\,\AAdAddAAd
+  \tfrac1{2} a\,(\tfrac 12 a^2+b)^2 \,\AAdAAddAd+a^3\, (\tfrac 12 a^2+b)\,\,\AAdAdAdAd    \label{eqGforests}
\end{eqnarray}  
Note that these $\mG$-forests consist of trees which are homogenous in the number of leaves ($\leftrightarrow Y$) but not in $a,b$ (unless powers of $b$ are counted twice). 
Upon setting $b = 0 $ we get $\mG^n (T, (a,0) ) = \mK^n (T, a) = a^n \mK^n(T)$ and get the first few $\mF$-forests: $\mK^{1}  = \Y  \equiv   \A$ and
\begin{eqnarray}
 \mK^{2}  &=&  \tfrac{1}{2} \Y \dm \Y  \equiv   \tfrac{1}{2} \AAd \nonumber \\
\mK^{3}  &=&   \tfrac{1}{2} (\Y \dm \Y) \dm \Y  \equiv   \tfrac{1}{2} \AAdAd  \label{ex:cum1to4} \\
 \mK^{4} &=&    \tfrac{1}{8}  (\Y \dm \Y)^{\dm 2}  + \tfrac{1}{2} ((\Y \dm \Y) \dm \Y) \dm \Y   \equiv   \tfrac{1}{8} \AAdAAdd  + \tfrac{1}{2} \AAdAdAd \  \nonumber \\
 \mK^5  &=& \qquad \dots \qquad = \tfrac{1}{4}\,\,\AAdAddAAd
+  \tfrac{1}{8} \,\AAdAAddAd+\tfrac{1}{2} \,\,\AAdAdAdAd 
    \end{eqnarray}
\begin{remark} \label{rmk:sym}The number of different tree shapes seen in both $\mG^n$, $\mK^n$ above is precisely the number of interpretations of an $(n-1)$-fold (commutative but not associative) diamond product, i.e. number of ways to insert parentheses. Starting from the empty tree, the resulting integers $\{ 0, 1, 1, 1, 2, 3, 6, ... \}$ are known as Wedderburn-Etherington numbers (OEIS A001190). The pre-factors in $2^{n} \mK^{n+1}$ further display the symmetry factors, e.g. $ 2^4 \, \mK^5 = 4 \,\AAdAddAAd + 2 \,\AAdAAddAd+ 8 \AAdAdAdAd\, $ with $4+2+8=14 = C_4$ where we recall that the $n.th$ Catalan number (A000108), standard example in analytic combinatorics, gives the number of binary trees with $n+1$ leaves. We note the combinatorial consistency of (\ref{eq:Krec}) with Segner's recursion $C_{n+1} = \sum_{k=0}^n C_k C_{n-k}$ if rewritten for $2^{n} \mK^{n+1}$.
\end{remark} 
 
\noindent The $2$-variate $\mF$-forests can be represented by all possible consistent ways of marking leaves with $\times$. This leads to  e.g. $2\times2$-matrix valued $\mK^{(2)}$ and $(\mathbb{R}^2)^{\otimes 4} \cong \mathbb{R}^{2^4}$ valued $\mK^{(4)}$, with representative trees of the form
$$   \AXd \qquad 
    \AXdAdAd  \;.  
$$
Tensor-contracting these 2-variate $\mK$-trees with powers of $a$ (number of $\A$ leaves) and $b$ (number of $\X$ leaves), and a $\X \rightsquigarrow \AAd$ substitution},\footnote{Strictly speaking, the $2$-variate case $A_T = (Y_T,  \angl{Y}_T)$ gives $\mathbb{E}_t A_T = (Y_t, \mathbb{E}_t \langle Y \rangle_T)$ with $\mathbb{E}_t \langle Y \rangle_T= (Y \diamond Y)_t (T) +\langle Y \rangle_t =  \AAd + BV$, where the bounded variation (BV) in $t$ component  $\langle Y \rangle_t$ shows up as a multiple of $z_2$ in (\ref{eq:AsymCFnew}), but does not contribute to subsequent diamond products.}
 
\bea
\mK^1  &=&  a\, \A+ b\, \AAd\nonumber\\
\mK^2 &=& \tfrac12\, \left(a\, \A +b  \, \AAd\right)^{\dm 2}= \tfrac 12 a^2\, \AAd + a b\, \AAdAd  +\tfrac 12 b^2\,\,\AAdAAdd \nonumber\\
\mK^3 
&=&  \tfrac 12 a^3\, \AAdAd + \tfrac12 a^2 b\,\AAdAAdd + a^2 b\, \AAdAdAd
+ a b^2\, \AAdAddAAd +    \tfrac12 a b^2\,\AAdAAddAd
+...
\nonumber\\
\mK^4 
&=&  \tfrac 12 a^4\, \AAdAdAd + \tfrac 1{2^3} a^4\,\AAdAAdd 
+ \tfrac 1{2} a^3 b\,\AAdAddAAd + \tfrac 1{2} a^3 b\, \AAdAAddAd
+ a^3 b\,\AAdAdAdAd +  \tfrac 1{2} a^3 b\,\AAdAddAAd +...\nonumber\\
\mK^5 
&=& \tfrac1{2} a^5\,\AAdAdAdAd + \tfrac 1{2^3} a^5\, \AAdAAddAd +\tfrac 1{2^2} a^5\, \AAdAddAAd + ...
\label{eq:KtreesNova}
\eea

\medskip

\begin{cor}[Forest reordering] \label{cor:RE}
The $\mG$-expansion (\ref{eqGforests}) is a reordering of the $2$-variate $\mK$-expansion applied to $A_T = (Y_T, \langle Y \rangle_T)$, as displayed in  (\ref{eq:KtreesNova}), based on the number of leaves. 
\end{cor}
\begin{proof}
From Theorem \ref{thm:mainintroG} we know that $\log \E_t \exp (a Y_T + b \langle Y \rangle_T)$ admits an (absolutely convergent) $\mG$ expansion, with terms homogenous in $Y$ ($\leftrightarrow$ number of leaves), and similarly a $2$-variate $\mK$-expansion with terms homogenous in $a,b$. The statement follows. 
\end{proof}

\begin{remark}[Generalized forests vs. cumulants] \label{rmk:EMart}
The $\mG$ and $\mK$ expansions coincide in absence of the term $b  \langle Y \rangle_T$: just set $b=0$ in (\ref{eqGforests}), (\ref{eq:KtreesNova}). In general, $b \ne 0$ matters in application (e.g. Section \ref{sec:SV}), the expansions are then different and the question arises about their qualitative difference. 
The $\mG$ expansion has the advantage of preserving some structural properties, lost in the $2$-variate $\mK$-expansion. To wit, consider the exponential martingale case  in which case
$$
 b =  - a^2/2 \quad \implies \quad 
\E_t \exp (a Y_T + b \langle Y \rangle_T) = \exp (a Y_t + b \langle Y \rangle_t) ,
$$
so that both $\mG$-sum, i.e.  $\sum_{k\ge2}   \mG_t^k$, and $\mK$-sum must vanish. However, while the zero $\mG$-sum only consists of zero summands $\mG^k$,
as seen in (\ref{eqGforests}) upon taking $b = - a^2/2$, this is not so for the $\mK$-sum, and only an infinite cascade of cancelations among the $\mF^k$'s causes the sum to vanish. 
\end{remark} 

\begin{remark} Corollary \ref{cor:RE} suggests an alternative proof of the $\mG$-expansion from the $\mK$-expansion, based on the combinatorics of forest reordering.
\end{remark}

\section{Proofs} 

\subsection{Proof of Theorem \ref{thm:mainintroG}}
 
Fix $z=(z_1,z_2)$ and define a family of semimartingales $Z=Z(\eps)$ given by
$$Z (\eps) := \eps z_1 Y + \eps^2 z_2 \langle Y \rangle.$$
The martingale $Y= \mathbb{E}_\bullet{A_T}$ inherits exponential integrability from $A_T$, as does $\langle Y \rangle$,
by a standard argument (e.g. along the proof of Novikov's criterion 
\cite[Ch.VIII.1.]{revuz2013continuous}.) In particular, for all $\eps$ in some (deterministic) neighbour around zero, the left-hand side of
$$
 \mathbb{E}_t e^{Z_{T} (\eps)} = e^{Z_t (\eps)+ \Lambda_t^T (\eps)}
$$
is well-defined, and so is $\Lambda_t^T (\eps)$, defined through this equality.  Clearly $\Lambda_T^T \equiv 0$, which shows that the process
$$
   \{ Z_{t} + \Lambda_t^T: 0 \le t \le T \} 
$$ 
is a stochastic logarithm, for all (small enough) $\epsilon$, omitted from notation, so that 
$$
\mathbb{E}_t  \Big( (Z_{t,T} + \Lambda_{t,T}^T) + \tfrac{1}{2} \langle Z + \Lambda^T\rangle_{t,T} \Big)= 0 \;.
$$
Here $Z_{t,T} = Z_T - Z_t,  \, \Lambda_{t,T}^T = \Lambda_{T}^T - \Lambda_{t}^T = - \Lambda_{t}^T$ and with diamond notation
\begin{equation} \label{key-equ}
         \Lambda_{t}^T =   \mathbb{E}_t Z_{t,T} + \tfrac{1}{2} \big( \langle Z + \Lambda^T) \dm ( \langle Z + \Lambda^T) \big)_t(T) \;.
\end{equation} 
By standard arguments, $\Lambda_t^T = \Lambda_t^T (\eps)$ is real analytic at zero\footnote{Seen by applying (conditional) dominated convergence, similar to the proof of real analyticity of the moment generating function; \cite{moran1984introduction,lukacs1970characteristic}.} with 
a.s. positive,  $\cF_t$-measurable radius of convergence and we write $\mathbb{G}^k_t (T,z) =: g_k$ for the ($\cF_t$-measurable) coefficients. Now insert
$$
          \Lambda_{t}^T = \sum_{k \ge 0} g_k \eps^k, \qquad \mathbb{E}_t Z_{t,T}  = \eps^2 z_2  \mathbb{E}_t  \langle Y \rangle_{t,T}, \qquad Z = \eps z_1 Y + \eps^2 z_2 \langle Y \rangle.              $$
into (\ref{key-equ}); match powers of $\eps$ to see $g_0=g_1=0$ and $g_k = \mathbb{G}^k$ as specified in (\ref{eq:Grec}). The multivariate extensions, part (iii), is straightforward and left to the reader. Part (i) goes along the lines of part (i) of Theorem \ref{thm:mainintroK} below,  noting that $A_T = Y_T \in L^{\infty-}$ (all moments) also implies $\langle Y \rangle_T \in L^{\infty-}$.

\subsection{Proof of Theorem \ref{thm:mainintroK}}

Part (ii) is an immediate corollary of Theorem \ref{thm:mainintroG}, applied with $(z,0) \in \mathbb{C}^2$, i.e. $z_2 = 0$. 

We give a proof of part (i) by localization, but also comment on a direct ``Hermite'' proof below. 
\begin{lemma} \label{lem:recwelldef}
Assume $A$ has $n$ moments, $n \in \mathbb{N}$. Then the recursion (\ref{eq:Krec}) 
is well-defined for $j \leqslant n$ and yields $(\mK^j_t(T): 0 \le t \le T)$ as a semimartingale with a $L^{n / j}$-integrable martingale part and a $L^{n / j}$-integrable bounded variation (BV) component. 
\end{lemma}

\begin{proof}
If $A \in L^n$ then also $M:= \mK^1_t := \E_t A \in L^n$. By the Burkholder-Davis-Gundy (BDG) inequality, $\sqrt{\langle M \rangle_T} \in L^n$ or 
$\langle M \rangle_T \in L^{n / 2}$. In case $n \geqslant 2$,
\[ 2\mK_t^2 (T)_{} = (\mK^1_{} \diamond \mK^1_{})_t (T)
 = \E_t \langle M \rangle_T - \langle M \rangle_t \in L^{n / 2} . \]
Call $M^{(2)}$ the martingale part of $\mK_{}^2$; clearly $M^{(2)}$
is a $L^{n / 2}$-martingale.
For $\mK_{}^3 =\mK^2 \diamond
\mK^1$ we first use Cauchy-Schwarz to estimate
\[ | \langle \mK^2, \mK^1 \rangle_T | = | \langle
 M^{(2)}, M \rangle_T | \leqslant \sqrt{\langle M^{(2)} \rangle_T}
 \sqrt{\langle M \rangle_T} . \]
By BDG, the 
right-hand side is the
product of random variables in $L^{n / 2}$ and $L^n$ respectively. Since
$\frac{1}{n / 3} = \frac{1}{n / 2} + \frac{1}{n}$ it follows immediately
from the (generalized) H\"older inequality that $\langle
\mK^2, \mK^1 \rangle_T \in L^{n / 3}$. Assume now $n
\geqslant 3$. Then $M^{(3)} := \E_{\bullet} \langle \mK^2,
\mK^1 \rangle_T$ is a well-defined $L^{n / 3}$-martingale which constitutes 
the martingale part of $\mK^3 = M^{(3)} - \langle
\mK^2, \mK^1 \rangle$. The general statement is that
\[ \mK^j = M^{(j)} + \tmop{BV} \]
with a $L^{n / j}$-martingale part and a bounded variation (BV) component, whenever $j = 1, \ldots n$. The same reasoning gives, by induction in $n$, the general statement. 
\end{proof}

\medskip

Given $A_T$, $\cF_T$-measurable with $N \in \mathbb{N}$ finite moments  --- but whose mgf is not necessarily finite --- we work with its two-sided truncation $(-\ell \vee A_T \wedge \ell)=: A_T^\ell$, followed by careful passage $\ell \to \infty$. 
Indeed, part (iii) applies to $A_T^\ell$ (bounded!) and hence shows that the $\mK^{\ell;n}$, $n=1,2,\dots$ defined by the recursion (\ref{eq:Krec}), started with $\mK^{\ell;1}= \E_\bullet A_T^\ell$, are well-defined and yield (up to a factorial factor) the $n$.th conditional cumulants of $A_T^\ell$. It is easy to see that the $n$.th (conditional) cumulant of $A_T$, which exists by Lemma \ref{lem:recwelldef} for $n \le N$, is the limit (a.s. and in $L^1$) of the corresponding cumulants for $A_T^\ell$, using the conditional dominated-convergence theorem.
It remains to be seen that the diamond recursion is also stable under this passage to the limit. 
The precise integrability properties of the $\mK$'s, obtained in Lemma \ref{lem:recwelldef} for $A_T$, are easily made uniform in the truncation parameter $\ell$; justification of taking $\ell \to \infty$ in the diamond recursion is then straightforward.  

The asymptotic expansion is then a straightforward consequence of validity of the expansion of part (i) for all integers $N$. This finishes the proof of (i).

\begin{remark}[Hermite]  A direct proof of part (i) that relates $\mK^n$, $n=1,2,3,...$ with the corresponding cumulants is possible based on Hermite polynomials. To understand the argument, start with the proof of part (ii), specialized to the cumulant case, so that the key identity reads 
$$
     \mathbb{E}_t e^{\eps Y_{T}} = e^{\eps Y_t + \Lambda_t^T (\eps)} \;.
$$
Rewritten with $Q^T_t (\eps) := -2 ( \mK^2_t (T) + \eps \mK^3_t (T) +
\eps^2 \mK^4_t (T) + \cdots)\,$,
\[ \E_t e^{\eps Y_T} = \E_t e^{\eps Y_T - \frac{\eps^2}{2}
Q_T^T (\eps)} = e^{\eps Y_t - \frac{\eps^2}{2} Q_t^T
(\eps)} \;,
\]
we can deduce, by definition of Hermite polynomials \cite[Ch.IV.3.]{revuz2013continuous}
martingality of
\[ e^{\eps Y_t - \frac{\eps^2}{2} Q_t^T (\eps)} =
\sum_{n \geqslant 0} \frac{\eps^n}{n!} H_n (Y_t, Q_t^T
(\eps)) \;. \]
By taking $(\partial / \partial \eps)^n |_{\eps = 0}$
we obtain a graded family of martingales, starting with ($n=2$)
\[ t \mapsto H_2 (Y_t, Q_t^T (0)) = Y_t^2 - \frac{1}{2} \mK^2_t (T) \;.\]
Applying It\^o's Formula over $[t, T]$ and taking $t$-conditional expectation
then identifies $\mK^2_t (T)$ correctly as $\E_t \langle M \rangle_{t, T} = (M \dm M)_t(T)$. Using
suitable relations between Hermite polynomials, see also \cite{ebrahimifard2018hopfalgebraic}, this argument extends to $n>2$ and provides an alternative 
route to the $\mK$-recursion, also well adapted to finite moment situations, but less elegant than our argument based on stochastic logarithms.
\end{remark} 

\section{Relation to other works}

We already commented in detail on Al\`os et al. \cite{alos2020exponentiation} and Lacoin et al. \cite{lacoin2019probabilistic}. 

Growing exponential expansions on trees are reminiscent of the Magnus expansion, a type of continuous Baker-Campbell-Hausdorff formula, with classical recursions based on rooted binary trees; Iserles--N{\o}rsett \cite{iserles1999solution}. And yet, the $\mF, \mK, \mG$ expansions are of a fundamentally different nature, for non-commutative algebra plays no r\^ole: our setup is one of multivariate random variables, associated martingales and their quadratic variation processes.

In a Markovian situation our expansion can be related to perturbative expansion of a ``KPZ'' type equation, by which we here mean a non-linear parabolic partial differential equation of HJB type. We make this explicit in the case when $A=f(B)$, for a Brownian motion $B$ and suitable $f$, in which case the $\mK$'s are described by a cascade of linear PDEs, detailed in Section \ref{sec:KPZ}, indexed by trees such as (\ref{ex:cum1to4}), in the exact same way as the ``Wild expansion" used in Hairer's KPZ analysis \cite{hairer2013solving}. (This link is restricted to the algebraic part of the expansions and rough paths, analytic renormalisation etc. play no r\^ole here.) 

That said, computing $\log \eef{ e^{\varepsilon A_T}}$ may also be viewed as a (linear) backward SDE with ``Markovian'' terminal data given by $e^{\varepsilon A_T}= e^{\varepsilon f(B_T)}$; upon suitable exponential change of variables this becomes  a quadratic BSDE as studied by Kobylanski \cite{kobylanski2000backward}, Briand-Hu \cite{briand2008quadratic} and many others, in the weakly non-linear regime (BSDE driver of order $\eps$). Yet another point of view comes from Dupire's functional It\^o-calculus \cite{dupire2019functional}  which would lead to similar (at least formal) computations as conducted in Section \ref{sec:KPZ}, for general $\cF_T$-measurable $A_T$. And yet another point of view comes from the Bou\'e--Dupuis \cite{boue1998variational}  Formula which gives an exact variational representation of 
$ \log \ee{e^{A_T}} $ when $A_T$ is a sufficiently integrable measurable function of Brownian motion up to time $T$; here (\ref{eq:ConvCGF}) can be viewed as an asymptotic solution to the Bou\'e--Dupuis variational problem in the weakly non-linear regime. %

In Section \ref{sec:WScum} we compare Theorem \ref{thm:mainintroK} to the work of Nourdin--Peccati \cite{nourdin2010cumulants} 
where the authors use Malliavin integration by parts to describe cumulants of certain Wiener functionals, and notably compute cumulants of elements in a fixed Wiener chaos. (The ability to work under exponential resp. (sub)exponential integrability assumptions is crucial to deal with second resp. higher order chaos.) An important element in the second chaos with explicit cumulants is provided by {\it L\'evy's stochastic area}, our (short) proof of its cumulant generating function should be compared with the combinatorial tour de force of \cite{levin2008combinatorial}, based on (signature) moments. It is conceivable that the multivariate cumulant formula applied to multidimensional Brownian motion and L\'evy area (a.k.a the Brownian rough path) provides new input into classical problems of stochastic numerics, \cite{gaines1997variable}.

In Section \ref{sec:SV} we apply Theorem \ref{thm:mainintroG} to establish a formula for the joint mgf of a process $X$, its quadratic variation $\angl{X}$, and $\eef{d\angl{X}_T / dT}$, quantities (a.k.a. log-price, total variance, forward variance)  that play an important role in stochastic financial modeling. Our expansion is most convenient for models written in forward variance form, state of the art in quantitative finance. In particular, the full expansion is computable in affine forward variance models, which includes the popular rough Heston model \cite{el2019characteristic}.

{\bf Acknowledgement.} PKF has received funding from the European Research Council (ERC) under the European Union's Horizon 2020 research and innovation programme (grant agreement No. 683164) and the DFG Excellence Cluster MATH+ (Projekt AA4-2).

\section{Examples}

\subsection{Brownian motion }

\begin{example}[Brownian motion with drift]
Let $A_t = \sigma B_t + \mu t $. Then 
$$ \mK^{1}_t (T) = \sigma B_t + \mu T = A_t + \mu (T-t), \ \ \mK^{2}_t (T)  = \frac{1}{2}  ( \mK^{1} \dm \mK^{1} )_{t} (T)  = \frac{1}{2} \sigma^2 (T-t) \; .$$
and $\mK^k \equiv 0$ for all $k \ge 3$. These are the cumulants of $A_T - A_t \sim N (\mu(T-t), \sigma^2 (T-t))$, as predicted by Theorem \ref{thm:mainintroG} (or Theorem \ref{thm:mainintroK}) and the $\mK$-forest expansion of the cumulant generating function (\ref{eq:ConvCGF}) is trivially convergent (with infinite convergence radius).

\end{example}

\begin{example}[Stopped Brownian motion]
Consider the martingale $A = B^\tau$, standard Brownian motion $B$ stopped at reaching $\pm 1$. We compute $$\mK^{1}_t(T) = \E_t B^\tau_T = B^\tau_t = B_{t \wedge \tau}, \ \  \mK^{2}_t(T) = \tfrac{1}{2} \E_t \angl{B^\tau}_{t,T} =  \tfrac{1}{2} \Big( \E_t (\tau \wedge T) - \tau \wedge t \Big), \ \ \dots \;. $$ The second quantity equals the conditional variance $\mathbb{V}_t (B^\tau_T) = \E_t (B^\tau_T)^2 - (B^\tau_t)^2$, and thus ``contains'' familiar identities from optional stopping. With $T=\infty$, $A_T = B_\tau$ takes values $\pm 1$ with equal probability. This is a bounded random variable, with globally defined and real analytic time-$t$ conditional cgf given by 
$$ \Lambda_t (x) = \log\Big( \tfrac{1}{2}[(1+B^\tau_t)e^x + (1-B^\tau_t)e^{-x}] \Big) \;.$$ 
Its convergence radius is random through the value of $B^\tau_t=B^\tau_t(\omega) \in [-1,1]$. For instance, when $t=0$, so that $B^\tau_t = 0$, we have $\Lambda_0 (x) = \log \cosh(x)$ with a $\mK$-forest expansion (\ref{eq:ConvCGF}) of finite convergence radius $\rho_0 = \pi / 2$. On the other hand, on the event $E := \{ B^\tau_t = \pm 1 \}$,  the cgf $\Lambda_t (x)$ trivially takes the value $\log e^{\pm x} = \pm x$ so that, on $E$, we have $\rho_t (\omega) = +\infty$.
\end{example}

\subsection{L\'evy area}\label{sec:levy}

We give a new proof of P. L\'evy's theorem, which compares favourably with other available proofs \cite{ikeda2014stochastic}, \cite{levin2008combinatorial}.

\begin{theorem}[P. L\'evy]
Let $\{X,Y\}$ be 2-dimensional standard Brownian motion and stochastic (``L\'evy") area be given by  
$$
\cA_t =\int_0^t \left(X_s\,dY_s - Y_s\,dX_s\right) \;.
$$
Then, for $T \in \big(-\tfrac{\pi}{2},\tfrac{\pi}{2}\big)$,
$$
\E_0\left[e^{\cA_T}\right] = \frac{1}{\cos T} =  \exp \left( \int_0^T\,\tan s\,ds \right) \;. 
$$
\end{theorem}
As a warmup, we compute the first few cumulants, using the $\mK$-recursion from Theorem \ref{thm:mainintroK}.
(We note $\langle \cA_T \rangle_T \notin L^\infty$, so that, strictly speaking, the result in \cite{lacoin2019probabilistic} is not applicable.)  
By a direct computation (or using a very special case of Theorem \ref{thm:signature}),
\beas
\mK^2  =\frac12\,\AAd
= \frac12\,(\cA \dm \cA)_t(T) &=&\frac12\,\int_t^T \left(\eef{X_s^2} + \eef{Y_s^2}\right)\,ds\\
&=& \frac{1}{2}\,(T-t)^2 +\frac12\,(X_t^2+Y_t^2)\,(T-t)
\eeas
\[
\implies \quad d\mK^2_s = (X_s\,dX_s+Y_s\,dY_s)(T-s) + \tmop{BV}.
\]
With $d\mK^1_s =  X_s\,dY_s - Y_s\,dX_s$ we see that the third forest vanishes,
\beas
\mK^3  = \mK^1 \dm \mK^2 = \eef{\int_t^T d\angl{\mK^1,\mK^2}_s} = \eef{\int_t^T \left[X Y\,d\angl{Y}_s-Y X\,d\angl{X}_s\right]\,(T-s)}= 0 \;.
\eeas

\begin{lemma} Set $J^{k}_t(T) := \frac{(T-t)^{k}}{k} + \frac12\,(X_t^2+Y_t^2)\, (T-t)^{k-1}$. Then
\[
\left(J^{j} \dm J^{k}\right)_t(T) = \frac{2}{j+k-1} \,J^{(j+k)}_t(T) \;.
\]
\end{lemma} 
\begin{proof}  With $dJ^{k}_s(T) = \left(X_s\,dX_s + Y_s\,dY_s \right)
\,(T-s)^{k-1}+\tmop{BV}$, computation as above. 
\end{proof} 
Note $\mK^n_t(T) = \alpha_n\,J^{n}_t(T)$ for $n=2,3$ with $\alpha_2=1,\alpha_3=0$. Assume by induction that this holds true up even/odd pair $(n-2,n-1)$, with $\alpha_{n-1} = 0$.
Then the cumulant recursion gives, with sum over even integers $j,k \ge 2$,
$$
2 \mK^{n} = 
\sum_{j+k = n} \alpha_j \alpha_k J^j \dm J^k
=  \sum_{j+k = n} \alpha_j \alpha_k \frac{2}{n-1} \,J^{n} \  =: 2 \alpha_n J^{n}  \ ,
$$
while $\mK^{n+1} = \mK^{1} \dm \mK^{n}$ (use $\mK^3, \dots, \mK^{n-1} = 0$), which vanishes for the same reason as $\mK^3$. (This completes the induction.) Hence
$$
       \alpha_n = \frac{1}{n-1} \big( \alpha_2 \alpha_{n-2} + \alpha_4 \alpha_{n-4} + \dots +   \alpha_{n-2} \alpha_2 \big), \qquad \alpha_ 2 = 1.
$$
Evaluated at $t=0$, using $J^{n}_0(T) = T^n /n$, we thus see the $\mK$-expansion take the form
$$
    \alpha_2T^2/2 + \alpha_4 T^4/4 + \alpha_5 T^6/6 + ... = T^2/2 + \tfrac{1}{3} T^4/4 + \tfrac{2}{15} T^6/6 + ...
$$
where one starts to see the integrated expansion of $\tan(T) = T +  \tfrac{1}{3} T^3 + \tfrac{2}{15} T^5 + ... $ integrated in time. To see that this is really so,
check that $f(T) := \sum_{j \in 2\mathbb{N}} \alpha_{j} T^{j-1}$ satisfies the ODE $f'(T)=f(T)^2+1$, with $f(0)=0$, which identifies $f \equiv \tan$. 

\subsection{Diamond products of iterated stochastic integrals} \label{sec:DPISI}

L\'evy's area is a particularly important example of Brownian iterated integrals, for which we have given explicit diamond computations.
We now present systematic diamond computations for iterated stochastic integrals,
which play a fundamental role in stochastic numerics and rough path theory \cite{kloeden1992numerical, lyons1998differential}. They are defined as follows.
For a word $a = i_1 \ldots .i_m$ of length $m$, with letters in $\mathbb{A}=
\{ i : 1 \leqslant i \leqslant d \}$, write $a i$ for the word (of length $m +
1$) obtained by concatenation of $a$ with the letter $i.$ Given a
$d$-dimensional Brownian motion $(B^i)$, introduce the iterated It\^o resp.
Stratonovich integrals
\[ 
B^{a i} = \int_0^{\bullet} B^a d B^i; \quad \widehat{B}^{a i} =
\int_0^{\bullet} \hat{B}^a \circ d B^i \;;
\]
set also $B^{\phi} = \hat{B}^{\phi}=1$ when $\phi$ is the empty word. One extends these definitions by linearity to linear combination of words, which becomes a commutative algeba under the shuffle product (e.g. $12 \shuffle 3 = 312+132+123$). Then the remarkable identity 
$$
\hat{B}_t^a\, \hat{B}_t^b = \hat{B}^c_t; \quad  c = a \shuffle b,
$$
holds true (and reflects the validity of the usual chain rule for Stratonovich integration). In contrast, resolving $B_t^a B_t^b$ requires quasi-shuffle (It\^o Formula) which we will not introduce here.
Let us also recall Fawcett's Formula (from e.g. Ch.3 in \cite{friz2014course})
\[ 
\E_0 \hat{B}_{0, 1}^a = \langle e_{\otimes}^{\mathbf{I}/ 2}, a \rangle
\backassign   \sigma_a \;.
\] 
\begin{theorem}\label{thm:signature}
Consider two (possibly empty) words $a, b$ with respective length $| a |, |
b |$ and letters $i = j \in \mathbb{A}.$ Then

\noindent (It\^o)
\[ (B^{a i} \diamond B^{b j})_t (T) = B_t^a\, B_t^b (T - t) + \frac{T - t}{1 +
 (| a | + | b |) / 2} (B^a \diamond B^b)_t (T) \]
\noindent  (Stratonovich)
\[ (\hat{B}^{a i} \diamond \hat{B}^{b j})_t (T) =\hat{B}_t^a\, \hat{B}_t^b (T
 - t)  + \hat{B}_t^a \sigma_b \frac{(T - t)^{\frac{| b |}{2} + 1}}{\frac{|
 b |}{2} + 1} + \hat{B}_t^b \sigma_a \frac{(T - t)^{\frac{| a |}{2} +
 1}}{\frac{| a |}{2} + 1} + \frac{T - t}{1 + \frac{| a | + | b |}{2}}
 (\hat{B}^a \diamond \hat{B}^b)_t (T) . \]
In case $i \ne j$ both diamond products vanish.
\end{theorem}

\begin{proof}
\noindent {\bf (It\^o)} By It\^o isometry, and the product rule $B_s^a B_s^b = B_t^a B_t^b + 
\ldots . + \langle B^{a i}, B^{b j} \rangle_{t, s} $, with omitted martingale increment $\int_t^s( B^a dB^b + B^b dB^a)$,
  \[ 
(B^{a i} \diamond B^{b j})_t (T) = \E_t \langle B^{a i}, B^{b j}
 \rangle_{t, T} 
 = \delta^{i j} \E_t \int_t^T B_s^a B_s^b d s 
 = \delta^{i j}
 \int_t^T (B_t^a B_t^b + (B^a \diamond B^b)_t (s)) d s 
 \]
From the scaling properties of Brownian motion,  the time $t$-conditional law of $(B^a \diamond
B^b)_t (s)$ is equal to the law of
\[
\left( \frac{s - t}{T - t} \right)^{\frac{| a | + | b |}{2}} (B^a
 \diamond B^b)_t (T), 
 \]
followed by an immediate integration over $s \in [t, T]$.

\noindent {\bf  (Stratonovich)} Note that
\[ 
\hat{B}^{a i} = \int \hat{B}^a d B^i + \tmop{BV} 
\]
so that, as in the It\^o case (but now with non-centered dots),
\[
(\hat{B}^{a i} \diamond \hat{B}^{b j})_t (T) 
= \delta^{i j} \,\E_t \int_t^T
 \hat{B}_s^a \hat{B}_s^b d s = \delta^{i j} \int_t^T ds\, (\hat{B}_t^a
 \hat{B}_t^b + \hat{B}_t^a \E_t \hat{B}_{t, s}^b + \hat{B}_t^b \E_t
 \hat{B}_{t, s}^a + (\hat{B}^a \diamond \hat{B}^b)_t (s)) . 
 \]
Using the (known) Stratonovich expected signature of Brownian motion,
\[ \E_t \hat{B}_{t, s}^b = (s - t)^{| b | / 2} \E_0 \hat{B}_{0, 1}^b = (s -
 t)^{{| b |}/{2}} \langle e_{\otimes}^{\mathbf{I}/ 2}, b \rangle = : (s -
 t)^{{| b |}/{2}} \sigma_b
\]
we see, with $i = j$,
\[ 
(\hat{B}^{a i} \diamond \hat{B}^{b i})_t (T) 
= \hat{B}_t^a \hat{B}_t^b (T
 - t) + \hat{B}_t^a \sigma_b \frac{(T - t)^{\frac{| b |}{2} + 1}}{\frac{|
 b |}{2} + 1} + \hat{B}_t^b \sigma_a \frac{(T - t)^{\frac{| a |}{2} +
 1}}{\frac{| a |}{2} + 1} + \frac{T - t}{1 + \frac{| a | + | b |}{2}}
 (\hat{B}^a \diamond \hat{B}^b)_t (T) .
 \]

\end{proof}

\begin{example}[Cameron--Martin formula]
Following \cite[Ch.XI.1]{revuz2013continuous}, the Laplace transform of $\int_0^1 B_s^2 ds$ is given by $$\left ( \cosh \sqrt{2 \lambda} \right)^{-1/2} = 
\exp (- \tfrac{1}{2} \lambda + \tfrac{1}{6} \lambda^2 - \tfrac{4}{45}\lambda^3  + ...) \;.$$
We can elegantly obtain this from the  $\mG$-expansion applied to the iterated It\^o integral $Y_t = \int_0^t B_s dB_s, \langle Y \rangle_1 = \int_0^1 B_s^2 ds$, so that $ \mG^2  = - \lambda (Y \dm Y)_t(T)$, and for $k>2$: 
$ \mG^k = \frac{1}{2} \sum_{j = 2}^{k -2} \mG^{k - j} \dm \mG^j  \, $.
A computation similar to the one given in the L\'evy area example, shows that the $n$.th cumulant is  given by $q_n/(2n)$, with recursion
$$
 q_n = \frac{2}{2 n - 1} (q_1 q_{n - 1} + \cdots + q_{n - 1} q_1), \qquad q_1 = 1 ,
$$
from which one can also obtain the explicit functional form. 
\end{example}

\subsection{Bessel process}
We use the $\mG$-expansion to establish some identities of the {\it Bessel square process} with (time-dependent) dimension $\delta = \delta (t) \ge 0$, given as solution to
$$
      d X_t = 2 \sqrt{X_t} d B_t + \delta (t) dt  \;.
$$ 
As in the case of L\'evy area, the $\mG$-expansion with diamond calculus compares (very) favourably with other available proofs, cf. \cite[Ch.XII]{revuz2013continuous}. 
For non-negative, bounded measurable $\mu = \mu (t)$, set $Y_T := \frac{1}{2} \int_0^T \sqrt{\mu} \,d X_s$, 
hence $dY_s = \frac{1}{2} \sqrt{\mu}\, d X_s$ so that the $\mG$-expansion gives the Laplace transform of the weighted Bessel average
\[ \langle Y \rangle_T = \int_0^T X_s \,\mu \,d s \;,\]
starting with
\[ \mG^2 = - \lambda \,Y \dm Y = \dots = 
-\lambda \left( X_t \int_t^T \mu (s) d s
+ \int_t^T \int_t^s \delta (r) \mu (s) d r d s \right), \]
followed by $\mG^3 = 0$. (Use $\E_t X_s = X_t + \int_t^s \delta (r) dr, t \ge s$.) By Lemma \ref{lem:cool} below,
\[ \log \E_t \exp \left( - \lambda \int_t^T X_s\, \mu_s\, d s \right) = 
\sum_{n \ge 2, \tmop{even}} \frac{(-
\lambda)^{n / 2}}{2}  \left( X_t \Gamma_n (t) + \int_t^T \delta (r)
\Gamma_n (r) d r \right) , \]
with $\psi (t) := \sum_{n \ge 2,\tmop{even}} (- \lambda)^{n / 2} \Gamma_n (t)$ rewritten as 
\[ \log \E_t \exp \left( - \int_t^T X (s) \mu (s) d s \right) =   \frac{1}{2} X_t \psi (t) + \frac{1}{2} \int_t^T \delta (r) \psi (r) d r \;.\]
Thanks to (\ref{GnODE}) $\psi$, is immediately identified as (unique) backward ODE solution to $ - \dot{\psi} = \lambda \dot{\Gamma}_2 + \psi^2 = - 2 \lambda \mu + \psi^2$ with terminal data $\psi(T) =0$. 
(This constitutes a new and elegant route to Cor 1.3, Thm 1.7 in \cite[Ch.XI.1]{revuz2013continuous}, therein only given for constant  $\delta$, in which case the ODE can be written as  $\phi'' =2 \lambda\mu \phi , \phi(t) = \exp (-\int_t^T \psi(s)ds )$.)

We can be more specific when $\mu$ is explicit. For instance, the conditional
Laplace transform of $X_T$ is obtained by taking $\mu (s) d s = \delta_T (d
s),$ justified by an easy approximation argument (e.g. $\mu^n = \frac{1}{n}
1_{[T - 1 / n], T}$), in which case the ODE becomes $ - \dot{\psi} =  \psi^2$ with terminal data $\psi(T) = - 2 \lambda$,
with unique solution $\psi(t) = -2\lambda / ( 1 + 2\lambda (T-t))$. Specializing to 
constant Bessel dimension $\delta (r) \equiv \delta$, $t = 0$, and $X_0 = x$, we obtain
\beas
\ee{\exp \left( - \lambda\, X_T \right)}&=& \exp \left( - \lambda \,x / (1 + 2 \lambda T) - \tfrac{\delta}{2} \log (1 + 2
   \lambda T) \right) \\
   &=& (1 + 2 \lambda T)^{- \delta / 2} \exp (- \lambda x /
   (1 + 2 \lambda T)) , 
\eeas
in agreement with \cite[Ch.XI.1]{revuz2013continuous}. (For what it's worth, the controlled ODE structure of $\dot{\psi}$, with additive noise $W := \Gamma_2$, makes sense for any deterministic c\`adl\`ag path $W$, and allows to compute the transform of any ``rough'' integral, $\int_0^T X_s (\omega) d W_s := X_{0,T} W_{0,T} - \int_0^T W_s^- dX_s$, with final integral in It\^o sense.)

\begin{lemma} \label{lem:cool}
The general even/odd pair in the $\mG$-expansion is of the form
\[ \mG_t^n = \frac{(- \lambda)^{n / 2}}{2}  \left( Z_t \Gamma_n (t) + \int_t^T
 \delta (r) \Gamma_n (r) d r \right), \quad \mG^{n + 1} \equiv 0, \]
with   
$\Gamma_n (t)$ determined by $\Gamma_2 (t)$=2$\int_t^T \mu (s) d s$ and
the recursion, for even $n \ge 4$,
\begin{equation} \label{GnODE}
- \dot{\Gamma}_n = \Gamma_2 \Gamma_{n - 2} + \Gamma_4 \Gamma_{n - 4} +
 \cdots + \Gamma_{n - 2} \Gamma_2, \quad \Gamma_n (T) = 0. \end{equation}
\end{lemma}   
\begin{proof}
The statement was seen to be correct for $n = 2$. Assume by induction that
it holds true, for all even/odd pairs up to $(n - 2, n - 1)$. In
particular then $d \mG^k = (- \lambda)^{k / 2} \Gamma_k \sqrt{Z} d B + d
(\mathrm{BV})$, for even $k <n$, and by the $\mG$ recursion, with sums below always over even
integers $j, k \ge 2$, 
    \[ 2 \mG^n = \sum_{j + k = n} \mG^j \diamond \mG^k = \sum_{j + k = n} \E_t
   \int_t^T d \langle \mG^j, \mG^k \rangle = (- \lambda)^{n / 2} \sum_{j + k =
   n} \E_t \int_t^T \Gamma_j (s) \Gamma_k (s) Z_s d s. \]
Set $\Gamma_n (t) \assign \sum_{j + k = n} \int_t^T \Gamma_j (s) \Gamma_k
(s) d s$ ($j, k, n$ even) and use $\E_t Z_s = Z_t + \int_t^s \delta (r) dr$ to conclude.($\mG^{n + 1} = 0$ is clear.)
The ODE statement is also immediate. 
\end{proof}

\subsection{A Markovian perspective and smooth KPZ} \label{sec:KPZ}

The previously encountered trees $\Big\{ \A, \AAd, \AAdAd , \AAdAAdd, \AAdAdAd \Big\}$ from Section \ref{sec:trees} were famously used in \cite{hairer2013solving} as a  minimal choice in indexing a {\it finite} expansion of the $(1+1)$-dimensional KPZ equation, with additional
analytical (rough path) arguments to deal with the remainder.\footnote{See also \cite{Gubinelli2016} and and \cite[Ch.15]{friz2014course} for similar trees in the KPZ context.} 
 The appearance of the same trees is more than a coincidence, as we shall now see. 
Consider functions $h_T = h_T (x)$ and $\xi = \xi (t, x)$ on $\R^d$ and $[0,T] \times \R^d$ respectively, for simplicity taken 
bounded with bounded derivatives of all orders, and consider
\begin{equation} \label{ATMarkov} A_T := h_T (B_T) + \int_0^T \xi (s, B_s) d s \end{equation}
with a standard $d$-dimensional Brownian motion $B$. Then
\[ \E_t e^{\lambda A_T (\omega) - \lambda \int_0^t \xi (s, B_s) d s} =
   \E_t e^{\lambda \left\{ h_T (B_T) + \int_t^T \xi (s, B_s) d s \right\}} =:
   e^{\lambda h (t, B_t)} =: z (t, B_t) \]
and $z = z (t, x)$ satisfies the Kolmogorov backward equation $- \dot{z} = \Delta z +
\lambda z \xi$, with terminal data $e^{\lambda h_T}$. By Cole-Hopf we obtain the $(d+1$)-dimensional KPZ equation
\begin{equation} \label{smoothKPZ}  \left( - \partial_t - \tfrac{1}{2} \Delta \right) h = \frac{\lambda}{2}
   (\nabla h \cdot \nabla h) + \xi, \quad h (T, \cdot) = h_T, 
\end{equation}
with smooth noise $\xi = \xi (t,x)$ and written in backward form. 
Following Hairer \cite{hairer2013solving}, who attributes such expansions to Wild (1955),
one has the (formal) tree indexed expansion\footnote{We use $|\lambda|$ to denote the number of leaves, which differs by $1$ from the number of inner nodes
which is the counting convention used in \cite[Equ (2.3)]{hairer2013solving}.} 
\begin{equation} \label{equ:KPZWildmartin}
h = u^{\A} + \lambda u^{\AAd} + 2 \lambda^2 u^{ \AAdAd} + \lambda^3 u^{\AAdAAdd} + 4 \lambda^3 u^{\AAdAdAd} \ \ + \dots  \ = \sum_{\tau} \lambda^{| \tau | - 1} u^{\tau} 
\end{equation}
with sum over all binary trees with $|\lambda| \ge 1$ leaves. 
More specifically, $u^{\A}$ is the unique (bounded) solution to the linear
problem $(\lambda = 0)$, and then, whenever $\tau = [\tau_1, \tau_2]$, the
root joining of trees $\tau_1$ and $\tau_2$, we get $u^{[\tau_1,\tau_2]} = u^{[\tau_2,\tau_1]}$ from\footnote{Cf. Remark \ref{rmk:sym} for related combinatorial comments, including symmetry factors.}
\begin{equation} \left( - \partial_t - \tfrac{1}{2} \Delta \right) u^{\tau} = \frac{1}{2}
   (\nabla u^{\tau_1} \cdot \nabla u^{\tau_2}), \qquad 2 u^{\tau} =  
   \mathcal{K}  \star (\nabla u^{\tau_1} \cdot \nabla u^{\tau_2}) =: 
   u^{\tau_1} \diamond u^{\tau_2}, \label{defDetDiamond}
  \end{equation}
where $\mathcal{K} \star (\ldots)$ denotes space-time convolution with the heat kernel.
(Thanks to our strong assumptions on forcing $\xi$ and terminal data $h_T$,
the recursion for the $u^{\tau} = u^{\tau} (t, x)$ is well-defined and all $u^\tau$ 
smooth.) We can then rewrite (\ref{equ:KPZWildmartin}) as
\[ \log \E^{t, x} e^{\lambda A_T (\omega)} = \lambda h (t, x) = \sum_{|
   \tau | \geqslant 1} \lambda^{| \tau |} u^{\tau} = \sum_{n \geqslant 1}
   \lambda^n \sum_{\tau : | \tau | = n} u^{\tau} = : \sum_{n \geqslant 1}
   \lambda^n K^n (t, x ; T) . \]
with $K^1 = u^{\A}$ and then recursively
\begin{equation} \label{normKrec} K^{n + 1} = \sum_{\tau : | \tau | = n + 1} u^{\tau} = \sum_{\cdots}
   u^{[\tau_1, \tau_2]} = \frac{1}{2} \sum_{\cdots} u^{\tau_1} \diamond
   u^{\tau_2} = \frac{1}{2} \sum_{i = 1^{}}^n K^i
   \diamond K^{n + 1 - i} , 
   \end{equation}
   using the unique decomposition of a binary tree $\tau$ with $| \tau | = n + 1 \ge 2$ \ leaves into smaller trees $\tau_1, \tau_2$ with $i$ (resp. $n+1 -i$) leaves for some $i=1,...,n$. 
By the Markov property, $\lambda^{- 1} \log \E_t e^{\lambda A_T (\omega)}
= h (t, B_t) \backassign \bar{h}_{},$ where $\overline{(\ldots)}$ indicates
composition of a function with time-space Brownian motion $\bar{B}_t = (t, B_t)$. Then the
$\bar{u}^{\tau_1}$ are semimartingales and by It\^o calculus 
\[ (\bar{u}^{\tau_1} \diamond \bar{u}^{\tau_2})_t (T) =\E_t \langle
   \bar{u}^{\tau_1}, \bar{u}^{\tau_2} \rangle_{t, T} =\E_t \left(
   \left( \int_t^T \nabla u^{\tau_1} \cdot \nabla u^{\tau_2} \right) (s, B_s)
   d s \right) =
   \bar{u}_t^{[\tau_1, \tau_2]} = (u^{\tau_1} \diamond u^{\tau_2}) \circ \bar{B}_t \]
which could be expressed as a commutative diagram. (Note that respective diamonds used on the left and right are different, introduced in Definition \ref{def:diamond} and  (\ref{defDetDiamond}) respectively.)
It
follows from (\ref{normKrec}) that $\bar{K}^n \assign K^n \circ \bar{B}$ satisfies the same diamond
recursion (\ref{eq:Krec}) as $\mathbb{K}^n$. The --- in view of (\ref{equ:KPZWildmartin}) still formal ---
conclusion
\[ \log \E_t e^{\lambda A_T (\omega)} = \lambda \int_0^t \xi (s, B_s) d s
   + \lambda u^{\A} (t, B) + \sum_{n \geqslant 2} \lambda^n \bar{K}^n (t,
   x ; T) \]
is then in exact agreement with the $\mathbb{K}$ expansion, since
\[ \mathbb{K}^1_t (T) \assign \E_t (A_T (\omega)) =\E_t \left(
   h_T (B_T) + \int_0^T \xi (s, B_s) d s \right) = u^{\A} (t, B_t) +
   \int_0^t \xi (s, B_s) d s \]
and subsequent terms in the recursion are not affected by the final BV term. Theorem \ref{thm:mainintroK} now settles convergence of (\ref{equ:KPZWildmartin}), with the additional advantage of removing the stringent conditions on the data: exponential moments for terminal data $h_T(B_T)$ and integrated forcing $\int \xi_0 ^T(s,B)s) ds)$ are enough.  We summarize this discussion as

\begin{theorem} \label{thm:KPZstuff}For $\lambda$ small enough, the perturbative expansion 
for the KPZ equation (\ref{smoothKPZ})
\begin{equation}  \label{eq:ConvCGFKPZ}
\lambda^{-1} \log \E^{t,x} e^{\lambda \left\{ h_T (B_T) + \int_t^T \xi (s, B_s) d s \right\}} = h(t,x) 
= \sum_{\tau}  \lambda^{|\tau|-1} u^{\tau} (t,x) 
  \end{equation} 
converges. Moreover, terms of same homogeneity have the stochastic representation 
$$
         \sum_{\tau : | \tau | = n} u^{\tau} (t,x) = \E \Big( \mK^n_t(T)  \, | \, B_t = x \Big)\;, 
$$ 
where $\mK^n_t ( T)$ follows (\ref{eq:Krec}), $n>1$, and started with $t$-conditional expectation of $$\mK^1_T (T) = A_T = h_T (B_T) + \int_0^T \xi (s, B_s) d s \;.$$  
\end{theorem}

\begin{remark} Theorem \ref{thm:KPZstuff} is really a Markovian perspective on the cumulant recursion and the above argument is readily repeated when Brownian motion (with generator $\Delta/2$) is replaced by a generic diffusion process (resp. its generator), in which case $(\nabla u^{\tau_1} \cdot \nabla u^{\tau_2}) /2$ in (\ref{defDetDiamond}) must be replaced by the corresponding {\it carr\'e du champ} $\Gamma (u^{\tau_1},u^{\tau_2})$, cf. \cite[Prop. VIII.3.3]{revuz2013continuous}. Sufficient conditions for the recursion (\ref{defDetDiamond}) to be well-defined , so that $u^\tau \in C^{1,2}$ for all $\tau$, hence $\bar u^\tau$ semimartingales, are a delicate issue.
The martingale based diamond expansion bypasses this issue entirely, with $\bar u^\tau$ as part of $\mK^{|\tau|}$ constructed directly, and so applies immediately 
when $B$ in (\ref{eq:ConvCGFKPZ}) is replaced by a generic diffusion processes on $\R^d$.
\end{remark}

\subsection{Cumulants on Wiener-It\^o chaos}  \label{sec:WScum}

On the classical Wiener space $C([0,T],\R)$, with Brownian motion $B(\omega,t) = \omega_t$, consider
an arbitrary element in the second Wiener It\^o chaos, written in the form
\[ A_T := I_2 (f) := \int_0^T \int_0^v f (w, v) d B_w d B_v \;, \]
with $f=f_A \in L^2$ on the simplex $\Delta_T = \{ (s,t): 0 \le s \le t \le T \}$. Note martingality $A_t := \mathbb{E}_t A_T$ so that $ \E_t A_{t, T} = \E_t A_T - A_t = 0$.
Then
\[ A_{t, T} = \int_t^T \int_0^v f (w, v) d B_w d B_v =
\int_t^T \int_t^v f (w, v) d B_w d B_v + \int_t^T \int_0^t f (w, v) d B_w d
B_v \]
and
\[ 
\langle A \rangle_{t, T} = \int_t^T \left( \int_0^v f (w, v) d B_w
\right)^2 d v = \int_t^T \left( \int_0^t f (w, v) d B_w + \int_{t}^v
(\ldots .) \right)^2\,dv,
\]
so that
\[ (A \diamond A)_t (T) = \mathbb{E}_t\langle A \rangle_{t, T} = \int_t^T \left(
\int_0^t f (w, v) d B_w \right)^2 d v + 0 + \int_t^T \int_{t}^v f^2 (w,
v) d w d v. \]
We have thus computed $\mK_t^2 (T)=\tfrac12 (A \diamond A)_t (T) .$ 
By
polarization, for $A = I_2 (f_A), C = I_2 (f_C),$
\[ (A \diamond C)_t (T) = \int_t^T \left( \int_0^t f_A (r, v) d B_r \right)
\left( \int_0^t f_C (r, v) d B_r \right) d v + \int_t^T \int_{t}^v f_A
(w, v) f_C (w, v) d w d v. \]
To go further, we exhibit the martingale part of $A \diamond C$ by writing
\[ \int_0^T \left( \int_0^t f_A (r, v) d B_r \right) \left( \int_0^t f_C (r,
v) d B_r \right) d v - \int_0^t (\ldots .) d v + \int_t^T \int_{t}^v f_A
(w, v) f_C (w, v) d w d v. \]
From the product rule, with $ {BV}_t = \int_0^t f_A (r, v) f_C (r, v) d
r$, we have
\[ \left( \int_0^t f_A (r, v) d B_r \right) \left( \int_0^t f_C (r, v) d B_r
\right) = \int_0^t \int_0^s [f_A (r, v) f_C (s, v) + f_C (r, v) f_A (s, v)]
d B_r d B_s +  {BV}_t \]
Letting $\otimes_1$ indicate integration in one (the right-sided) variable and tilde
symmetrisation,
\[ f_A  \tilde{\otimes}_1 f_C := \int_0^T ([f_A (r, v) f_C (s, v) + f_C
(r, v) f_A (s, v)]) d v \]
so that
\[ (A \diamond C)_t (T) = \int_0^t \int_0^s (f_A  \tilde{\otimes}_1 f_C) d
B_r d B_s +  {BV}_t (T) \]
with $$(A \diamond C)_0 (T) =  {BV}_0 (T) = \int_0^T \int_{0}^v f_A
(r, v) f_C (r, v) d r d v = \langle f_A, f_C \rangle_{\Delta_T}
=: f_A \otimes_2 f_C.
$$
In particular, we see that from (\ref{ex:cum1to4}) that the third cumulant of $A_T = I_2 (f_A)$ is given by
\[ \kappa_3 (A_T) = \kappa_3 (I_2 (f_A)) =3 (A \diamond (A \diamond A))_0 (T) = \langle f_A, f_{A \diamond A}
\rangle = \langle f_A, (f_A  \tilde{\otimes}_1 f_A) \rangle. 
\]
Theorem \ref{thm:mainintroK} then provides, in the present setting, an alternative to the (Malliavin calculus based) approach of Nourdin--Peccati \cite{nourdin2010cumulants}: by (5.22) in that paper, the $n$.th cumulant of $I_2 (f)$ is given by some explicit formula which reduces to (in case $n=3$) our formula.
It is not difficult to push this ``diamond'' computation to recover cumulants for general integer $n$. The diamond approach of course works just as well for higher Wiener-It\^o chaos and $d$-dimensional Wiener space, as was already seen in Section \ref{sec:DPISI}. Note however that the exponential integrability assumed in part (ii) of Theorem \ref{thm:mainintroK}, valid in the second chaos, does not hold for third and higher chaos. However, any fixed chaos random variable has moments of all orders so that part (i) of this theorem is applicable. 
Last not least, note that \cite{nourdin2010cumulants} deal with Gaussian fields, whereas we have been dealing with processes.

\subsection{Stochastic volatility} \label{sec:SV}

\subsubsection{Joint law of SPX, realized variance and VIX squared}

We return to the financial mathematics context that originally gave rise to diamond expansions result. Our framework permits the valuation and hedging of complex derivatives involving combinations of assets and their quadratic variations. To be specific, let $S$ be a strictly positive continuous martingale. Then $X := \log S$ is a semimartingale, with $e^X_T \in L^1$, so that $X_T$ has moments of all orders. 
If the quadratic variation process $\angl{X}$ is absolutely continuous, the  {\it stochastic variance} and {\it forward variance} are given by
$$
v_t := d \angl{X}_t/ dt \;, \quad \xi_t(T) = \eef{v_T} \;.
$$
Upon integration in time quantities these quantities - realized and expected quadratic variation at a future time $T$ - constitute the payoff of a variance swap and VIX$^2$ respectively. (This application entails the interpretation of $e^X$ as the risk neutral price process of the SPX index on which the VIX index is built.) We now illustrate the use of Theorem \ref{thm:mainintroG} to determine the joint law of (log)-price, realized and expected quadratic variation at a future time $T$, the precise setting for consistent pricing of options on SPX, realized variance and VIX squared, with time-$T$ payoff given by
\[
\zeta_t(T) =  \int_T^{T+\Delta}\,\xi_t(u)\,du= \eet{\int_T^{T+\Delta}\,v_u\,du} = \E_t \angl{X}_{T,T+\Delta} .
\]

\begin{theorem} \label{thm:TripleJointMGF} Assume $X_T$ has exponential moments. Then for $a, b, c \in \RR$ sufficiently small,
\begin{equation} \label{generalCF}
\eef{e^{a\,X_T+ b\,\angl{X}_{T} +c\,\zeta_T(T)}} 
=\exp\left\{a\, X_t + b\,\angl{X}_{t} +  c\,\zeta_t(T)+\sum_{k=2}^\infty\,\mG^k_t (T;a,b,c)  \right\} \;,
\end{equation}
where the $\mG^k$'s are given recursively by (\ref{eq:Grec}),  starting with 
$$
           \mG^2 = \left( \tfrac 12 a(a-1)  + b \right) X \dm X + a c \,X \dm \zeta + \tfrac{1}{2}c^2\,\zeta \dm \zeta \ .
$$ 
\end{theorem}
\begin{proof} This is a direct consequence of multivariate $\mG$-expansion of Theorem \ref{thm:mainintroG}, employed with time-$T$ terminal data, re-expressed in terms of the martingale $Y = X - \tfrac{1}{2} \angl{X} $,
$$
  a\,X_T+ b\,\angl{X}_{T}+c\,\zeta_T(T) = a\,Y_T+ \left(b-\tfrac12a\right) \angl{Y}_{T}+c\,\zeta_T(T)   \;.
$$
We note for later use that the convergent $\mG$-sum is exactly equal to $\Lambda$, which satisfies the ``abstract Riccati'' equation (\ref{key-equ}), 
\[
  \Lambda_{t}^T =   \mathbb{E}_t Z_{t,T} + \tfrac{1}{2} \big( \langle Z + \Lambda^T) \dm ( \langle Z + \Lambda^T) \big)_t(T) .
\]
Then, since $Y$ and $\zeta$ are martingales, and $\angl{Y}$ is $BV$ hence annihilated by $\dm$,
\beas
  \Lambda &=&\left(b-\tfrac12a\right) Y \dm Y+  \tfrac{1}{2} \big( \langle  a\,Y+c\,\zeta+ \Lambda) \dm ( \langle a\,Y+c\,\zeta + \Lambda) \big)\\
&=&  \left(\tfrac12 a^2+b-\tfrac12a\right) Y \dm Y + a c\,Y\dm \zeta + a\,Y\dm \Lambda+  \tfrac{1}{2} \big( \langle c\,\zeta+ \Lambda) \dm ( \langle c\,\zeta + \Lambda) \big)\\
  &=&  \left(\tfrac 12 a( a-1)+b\right) Y \dm Y + a c\,Y\dm \zeta + a\,Y\dm \Lambda + c\, \zeta \dm \Lambda +  \tfrac{1}{2} c^2\,\zeta\dm \zeta +   \tfrac{1}{2} \Lambda \dm \Lambda,
\eeas
which in terms of pictures with $Y=\A$ and $\zeta=\Z$ gives
\begin{equation}\label{eq:roughRic}
\Lambda  = \left(  \tfrac 12 a(a-1) + b\right) \AAd 
+ a c\,\AZd + a\,\A \dm \Lambda + c\, \Z \dm \Lambda +  \tfrac{1}{2} c^2\,\ZZd +   \tfrac{1}{2} \Lambda \dm \Lambda\;.
\end{equation} 
 

\end{proof}

\begin{remark} Martingality of $S=e^X$ is seen in (\ref{generalCF}) upon setting $(a,b,c)=(1,0,0)$. It is a feature of the $\mG$ expansion that  each term $\mG^k_t(T;1,0,0)$ vanishes so that the martingality constraint is preserved at arbitrary truncation of the $\mG$ expansion, reminiscent of (Lie algebra preserving) Magnus expansions for differential equations on Lie groups \cite{iserles1999solution}. This is not be the case for the multivariate $\mK$ expansion, cf. Remark \ref{rmk:EMart} for a related discussion.
\end{remark} 

\begin{remark} We insist that (\ref{generalCF}) is a model free result, with $\mG$-expansion given naturally by diamond trees with two typos of leaves corresponding to $X$ and $\zeta$. 
\end{remark} 

\subsubsection{Forward and affine forward variance models}

After Black--Scholes (constant volatility), classical stochastic volatility models consider $v=v(t,\omega)$ as stochastic process in its own right. ``Third generation models'' where one specifies directly forward variances - viewed as a family of martingales indexed by their individual time horizon -  
are nowadays ubiquitous in equity financial modeling. In full generality this reads
\begin{equation} \label{xigeneral}
d_t \xi_t(u) = d_t \eef{v_u} =:  \sigma_t{(u)}  \, dW_t^u \;, \qquad t \le u \;,
\end{equation}
where $v_t \equiv \xi_t (t)$ and $dS_t / S_t = \sqrt{v_t} dZ_t$, where the correlation (covariation) structure all the Brownian family $\{Z, W^T: T \ge 0 \}$ also needs to be specified. We can then immediately write (diamonds with $\zeta$ amount to average diamonds with $\xi(T')$ over $T' \in [T,T+\Delta]$)
\beas
\AAd \ \  &:=& (X \dm X)_t(T) = \eef{\int_t^T\,d\angl{X}_s} = \int_t^T\,\xi_t(s)\,ds\\
\AZd  \ \ &:=& (X \dm \xi(T'))_t(T) = \eet{\int_t^T\,d\angl{X,\xi(T')}_s}= \E_t \int_t^T \sqrt{v_s}\, \sigma_s(u)  d \langle Z, W^{T'} \rangle_s \\
\ZZd  \ \ &:=& (\xi(T')\dm \xi(T'))_t(T) = \eet{\int_t^T\,d\angl{\xi(T'),\xi(T')}_s}=\int_t^T\, \eet{ \sigma_s(T')^2} \, ds .
\eeas
At this stage, more structure is required for computations. 
A particularly simple choice is the affine specification $\sigma_t(u)  = \kappa(u-t)\,\sqrt{v_t}$ of \cite{gatheral2019affine}:
\bea
\frac{dS_t}{S_t}&=&\sqrt{v_t}\,dZ_t\nonumber\\
d_t \xi_t(u)&=&  \kappa(u-t)\,\sqrt{v_t}\,dW_t \;, \qquad t \le u \;,
\label{eq:AFV}
\eea
where $\kappa$ is some $L^2$-kernel and the Brownian drivers satisfy $d\angl{W,Z}_t/dt = \rho$. 
Note that $$ \xi_t(u) = \xi_0(u) + \int_0^t  \kappa(t-s)\,\sqrt{v_s}\,dW_s \;,\qquad v_t \equiv \xi_t(t) = \xi_0(t) + \int_0^t  \kappa(t-s)\,\sqrt{v_s}\,dW_s $$
so that stochastic variance solves a stochastic Volterra equation.
 Special cases are the Heston and rough Heston models with exponential and power-law kernels respectively. We also note that
\begin{equation} \label{eq:AFV2}
      d_t \zeta_t (T) = \left( \int_T^{T+\Delta}  \kappa(u-T) \,d u \right) \,\sqrt{v_t}\,dW_t =: \bar{\kappa}(T-t) \,\sqrt{v_t}\,dW_t
\end{equation}
has the same form as (\ref{eq:AFV}); in computations, $\zeta$ can then effectively be replaced by $\xi$. 

\begin{lemma}\label{lem:AffineTrees}

In the affine forward variance model (\ref{eq:AFV}) all diamond trees (with leaves of two types $X = \A$ and $\zeta = \Z\,$, respectively), and hence all forests terms
$\mG^k_t$ in (\ref{generalCF}) are of the form 
\beq
\int_t^T\,\xi_t(u)\,h(T-u)\,du
\label{eq:convolutionForm}
\eeq
for some integrable function $h$.
\end{lemma}

\begin{proof} As above, $\AAd = \int_t^T\,\xi_t(s)\,ds$, but now with (\ref{eq:AFV}), also noting (\ref{eq:AFV2}),
\beq
\AZd \ \ \  
=  \rho\,\int_t^T\,\xi_t(u)\,\bar \kappa(T-u)\,du \;, 
\qquad \ZZd \ \ \  
=  \int_t^T\,\xi_t(u)\,\bar \kappa(T-u)^2\,du \;.
\label{eq:AFV3}
\eeq
We thus see that the claim holds for all diamond trees with two leaves and proceed by induction.
Consider two trees 
\[
\mT^i_t =\int_t^T\,\xi_t(u)\,h^i(T-u)\,du, \qquad i = 1,2
\]
of the supposed form. 
Then
\beas
\left(\mT^1 \dm \mT^2 \right)_t(T) &=& \eef{\int_t^T\,d\angl{\mT^1,\mT^2}_u}\\
&=&\eef{\int_t^T \int_u^T \int_u^T\,\,h^1(T-s) \,h^2(T-r) \, \,ds \, dr \,d\angl{\xi(s),\xi(r)}_u} \\
&=&\eef{\int_t^T\,v_u \,\kappa(s-u)\,\kappa(r-u)\,du\,\int_u^T\,h^1(T-s)\,ds\,\int_u^T\,\,h^2(T-s)\,dr\,}\\
&=&\int_t^T\,\xi_t(u)\,h^{12} (T-u) \,du,
\eeas
and the induction step is completed upon setting 
\[
h^{12}(T-u)  = \int_u^T\,h^1(T-s)\,\kappa(T-s)\,ds\,\int_u^T\,h^2(T-r)\,\kappa(T-r)\,dr.
\]
\end{proof}

\begin{remark} 
The statement and proof of Lemma \ref{lem:AffineTrees} may be extended to the non time-homogeneous setting 
$ d_t \xi_t(u) = \kappa(u,t)\,\sqrt{v_t}\,dW_t $ without much extra effort.
\end{remark}

\begin{example}[Classical Heston]
In this case,
\[
d\xi_t(u) = \nu\,e^{-\lambda\,(u-t)}\,\sqrt{v_t}\,dW_t.
\]
Then, for example,
\beas
\AAdAd =(X \dm (X \dm X))_t(T) &=&\frac{\rho\,\nu}{\lambda}\,\int_t^T\,\xi_t(u)\,\left[1-e^{-\lambda(T-u)}\right]\,du.
\eeas

\end{example}

\begin{example}[Rough Heston]
In this case, with $\alpha  = H + 1/2 \in ( 1/2, 1)$, 
\[
d\xi_t(u) = \frac{\nu}{\Gamma(\alpha)}\,(u-t)^{\alpha-1}\,\sqrt{v_t}\,dW_t.
\]
Then, for example,
\beas
\AAd = (X \dm X)_t(T) &=&\int_t^T\,\xi_t(u)\,du,
\\
\AAdAAdd = ((X \dm X) \dm (X \dm X))_t(T) &=& \frac{\nu^2}{\Gamma(\alpha)^2}\,\int_t^T\,\xi_t(u)\,du\,\left(\int_u^T\,(s-u)^{\alpha-1}\,ds\right)^2\\
&=& \frac{\nu^2}{\Gamma(1+\alpha)^2}\,\int_t^T\,\xi_t(u)\,(T-u)^{2\,\alpha}\,du.\\
\eeas
For a bounded forward variance curve $\xi$ one then sees that diamond trees with $k$ leaves are of order $(T-t)^{1+(k-2) \alpha}$.
In this case, the $\tmF$-expansion (forest reordering according to number of leaves) has the interpretation of a short-time expansion, the concrete powers of which depend on the roughness parameter $\alpha = H + 1/2 \in (1/2,1)$. The resulting diamond expansions (which can obtained by alternative methods in the rough Heston case)  were seen to be numerically efficient in \cite{callegaro2020rough, gatheral2019rational}.
\end{example}
At this stage it is tempting to combine Lemma \ref{lem:AffineTrees} with Theorem \ref{thm:TripleJointMGF} to compute the triple-joint mgf of $X_T$, $\angl{X}_{t,T}$, and $\zeta_T(T)$ by summing the full $\mG$-expansion for an affine forward variance model.  We then see that the mgf is necessarily of the convolutional form
\[
\log \eef{e^{a\,X_T + b\,\angl{X}_{T}+ c\,\zeta_T(T)}} =a\,X_t + b\,\angl{X}_{t} + c\,\zeta_t(T)+ \int_t^T \xi_t(u) \,g(T-u;a,b,c, \Delta) du,
\]
which amounts to an infinite-dimensional version of the classical affine ansatz.  
Inserting
$ \Lambda_{t} (T) = \int_t^T \xi_t(u) g(T-u;a,b,c, \Delta) du $
directly into the ``abstract Riccati'' equation \eqref{eq:roughRic}, we 
%
 readily obtain that the triple-joint MGF satisfies a convolution Riccati equation of the type considered in \cite{jaber2019affine, gatheral2019affine}.
We summarize this in the following theorem.
\begin{theorem}\label{thm:triplejointMGF}
Let
\beas
dX_t &=& -\tfrac 12 \,v_t\,dt + \sqrt{v_t}\,dZ_t\\
d\xi_t(T) &=& \kappa(T-t)\,\sqrt{v_t}\,dW_t,
\eeas
with $d\angl{W,Z}_t=\rho\,dt$ and let $\angl{X}_{t,T} = \angl{X}_T-\angl{X}_t$.  Further let $\tau = T-t$,
$
\bar \kappa(\tau) = \int_\tau^{\tau+\Delta}\,\kappa(u)\,du,
$
and define the convolution integral
$$
(\kappa \star g) (\tau) = \int_0^\tau\,\kappa(\tau - s)\,g(s)\,ds.
$$
Then
\[
\eef{e^{a\,X_T + b\,\angl{X}_{t,T}+ c\,\zeta_T(T)}} =\exp\left\{a\,X_t +c\,\zeta_t(T)+ (\xi \star g) (T-t;a,b,c,\Delta)\right\}
\]
where $g(\tau;a,b,c,\Delta)$ satisfies the convolution Riccati integral equation
\beq
g(\tau;a,b,c,\Delta)=b -\tfrac12 a + \tfrac12\,(1-\rho^2)\,a^2+ \tfrac12\,\left[\rho a + c \, \bar \kappa(\tau) + (\kappa \star g)(\tau;a,b,c,\Delta)\right]^2,
\label{eq:jointRiccati}
\eeq
with the boundary condition
$
g(0;a,b,c,\Delta)= b+\tfrac 12\,a(a-1) +\rho a c\,\bar \kappa (0)+ \tfrac12\,c^2\,\bar \kappa(0)^2
$.

\end{theorem}
\begin{proof}
From \eqref{eq:AFV2},
$
d\zeta_t(T) 
= \sqrt{v_t}\,dW_t\,{\bar \kappa}(T-t)
$.
As before, let 
$ \Lambda_{t}  = \int_t^T \xi_t(u) g(T-u;a,b,c, \Delta) du $.  
Then dropping the arguments $a,b,c,\Delta$ for ease of notation,
\beas
d\Lambda_t&=& -\xi_t(t)\, g(T-t)\,dt + \int_t^{T}\,d\xi_t(s)\, g(T-s)\,ds\\
&=& -v_t\, g(T-t)\,dt + \sqrt{v_t}\,dW_t\,\int_t^{T}\,\kappa(s-t)\, g(T-s)\,ds\\
&=& -v_t\, g(T-t)\,dt + \sqrt{v_t}\,dW_t\,(\kappa \star g)(T-t).
\eeas
We compute
\beas
d\angl{X}_t
&=& v_t\,dt\\
d\angl{X,\zeta}_t &=& \rho\,v_t\, \bar \kappa(T-t)\,dt\\
d\angl{\zeta}_t &=&  v_t\,\bar \kappa(T-t)^2\,dt\\
d\angl{X,\Lambda}_t &=&  \rho\,v_t\,(\kappa\star g)(T-t)\, dt \\
d\angl{\Lambda}_t &=& v_t\, \left[ (\kappa\star g)(T-t) \right]^2\,dt\\
d\angl{\zeta,\Lambda}_t &=& v_t\, \bar \kappa(T-t)\,(\kappa\star g)(T-t)\,dt.
\eeas
Integrating these terms from $t$ to $T$, followed by taking a time-$t$ conditional expectation allows us to compute all diamond products in the ``abstract Riccati'' equation \eqref{eq:roughRic}
\[
\Lambda  = \left(  \tfrac 12 a(a-1) + b\right) \AAd 
+ a c\,\AZd + a\,\A \dm \Lambda +  \tfrac{1}{2} \,\left[c\,\Z +  \Lambda \right]^{\dm 2}
\]
to yield
\[
g(\tau)=\left(  \tfrac 12 a(a-1) + b\right) + a\,\rho\,\bar \kappa(\tau) +a\,\rho\, (\kappa \star g)(\tau)
+\tfrac12\,
\left[ c \, \bar \kappa(\tau) + (\kappa \star g)(\tau)\right]^2,
\]
which upon rearrangement gives \eqref{eq:jointRiccati}.
Finally, $(\kappa \star g)(0) =0$ gives the boundary condition
\beas
g(0) = b+\tfrac 12\,a(a-1) +\rho a c\,\bar \kappa (0)+ \tfrac12\,c^2\,\bar \kappa(0)^2.
\eeas

\end{proof}

\bibliographystyle{alpha}
\bibliography{Cumulants}

\appendix

\end{document}